\documentclass[a4paper,reqno,12pt]{amsart}
\usepackage[english]{babel}
\usepackage[T1]{fontenc}
\usepackage[utf8]{inputenc}
\usepackage{mathptmx,amsmath,amssymb}
\usepackage{fullpage,microtype}
\usepackage{enumitem}
\usepackage[colorlinks=true,linkcolor=blue,citecolor=blue,urlcolor=blue]{hyperref}
\hypersetup{pdftitle={Minimal m-subharmonic functions with nonmaximal Hessian measures},pdfauthor={Thai Duong Do and Van Thien Nguyen}}
\newtheorem{theorem}{Theorem}[section]
\newtheorem{proposition}[theorem]{Proposition}
\newtheorem{lemma}[theorem]{Lemma}
\newtheorem{corollary}[theorem]{Corollary}
\theoremstyle{definition}

\theoremstyle{remark}
\newtheorem{remark}[theorem]{Remark}
\numberwithin{equation}{section}
\newcommand{\C}{\mathbb C}
\newcommand{\R}{\mathbb R}
\newcommand{\D}{\mathbb D}
\newcommand{\E}{\mathcal E}
\newcommand{\F}{\mathcal F}
\newcommand{\PSH}{\operatorname{PSH}}
\newcommand{\SH}{\operatorname{SH}}
\newcommand{\supp}{\operatorname{supp}}
\newcommand{\ddc}{dd^c}

\newcommand{\vol}{\operatorname{Vol}}
\newcommand{\one}{\mathbf 1}
\setlist[enumerate]{label=\textup{(\roman*)},leftmargin=*,itemsep=3pt}
\allowdisplaybreaks[1]

\begin{document}
\author{Thai Duong Do\textit{$^{1,2}$}}
\address{$^{1}$Institute for Artificial Intelligence, VNU University of Engineering and Technology, Hanoi, Vietnam.}
\address{$^{2}$Department of Mathematics, National University of Singapore, 10 Lower Kent Ridge Road, 119076, Singapore.}
\email{dtduong@vnu.edu.vn, dtduong@nus.edu.sg}
\author{Van Thien Nguyen\textit{$^{3}$}}
\address{$^{3}$FPT University, Education zone, Hoa Lac high tech park, Km29 Thang Long highway, Thach That ward, Hanoi, Viet Nam.}
\email{thiennv15@fe.edu.vn}
\title[Minimal functions and nonmaximal measures]{Minimal $m$-subharmonic functions with nonmaximal Hessian measures}

\begin{abstract}
For every $2\le m\le n$, we construct H\"older continuous functions on the closed unit ball that are minimal in the Cegrell class $\F_m$, although their Hessian measures are not maximal in the $m$-subharmonic ordering. This answers a question posed by the second author. We also obtain a minimality criterion based on partial pluricomplex energy and an explicit family of Monge--Amp\`ere examples on a product domain. For $m=1$, minimality and maximality are equivalent on a ball.
\end{abstract}
\maketitle
\enlargethispage{2pt}

\section{Introduction and main results}

The complex Hessian operators form a family between the Laplacian and the complex Monge--Amp\`ere operator. Their admissible functions are the $m$-subharmonic functions, with
\[
\PSH(\Omega)=\SH_n(\Omega)\subset\cdots\subset\SH_1(\Omega)=\SH(\Omega),
\]
where $\Omega\subset\C^n$ is a domain. We use the normalization
\begin{equation}\label{normalization}
d^c=\frac{i}{2\pi}(\bar\partial-\partial),\qquad
\ddc=\frac{i}{\pi}\partial\bar\partial,
\end{equation}
so that $\ddc\log|\zeta|=\delta_0$ in one variable. Set $\beta=\ddc|z|^2$ and
\[
H_m(u)=(\ddc u)^m\wedge\beta^{n-m}.
\]
A superscript $-$ denotes nonpositive functions, and all hyperconvex domains in this paper are bounded.

Bedford--Taylor theory~\cite{BT82} defines the Monge--Amp\`ere operator for locally bounded plurisubharmonic functions. B\l ocki~\cite{Bl05} developed the corresponding theory for complex Hessian operators. For unbounded functions, Cegrell's classes~\cite{Ce98,Ce04} and their Hessian counterparts studied by Lu~\cite{Lu15} allow the operators to be defined by decreasing approximation. The classes $\F_m$ and $\E_{1,m}$ control total Hessian mass and energy, respectively. We recall their definitions in Section~\ref{sec:hessian}. When $m=n$, the subscript $m$ is omitted.

Several aspects of these classes are relevant to the relation between functions and their Hessian measures. The second author studied vector spaces generated by finite-energy $m$-subharmonic functions in~\cite{Ng16}, and Czy\.{z} and the second author examined the constants in mixed energy estimates in~\cite{CN17}. The class $\mathcal N_m$, studied in~\cite{Ng19}, describes functions whose smallest $m$-maximal majorant is zero. Boundary behavior of functions in $\F_m$ was investigated by Nguyen Quang Dieu and the first author~\cite{ND22}, who obtained estimates for the volumes of sublevel sets and sufficient conditions for membership in $\F_m$. The present authors studied weighted $m$-energy classes, their relation to Cegrell classes, and associated integration and convergence results in~\cite{DN23}.

We consider a question about order and total mass. Pointwise inequalities between functions in $\F_m$ need not give inequalities between their Hessian measures as measures. They do give inequalities after integration against negative $m$-subharmonic functions~\cite[Proposition~2]{Ng18}. This leads to the following ordering.

Bengtson~\cite{Be17} introduced an ordering of measures using negative plurisubharmonic test functions. The second author extended this ordering to $m$-subharmonic functions in~\cite{Ng18}. Let $\Omega$ be a bounded $m$-hyperconvex domain. For finite positive Borel measures on $\Omega$, set
\begin{equation}\label{order}
\nu\succeq_m\mu
\quad\Longleftrightarrow\quad
\int_\Omega(-\varphi)\,d\nu\ge
\int_\Omega(-\varphi)\,d\mu
\quad\text{for all }\varphi\in\E_{0,m}(\Omega)\cap C(\Omega).
\end{equation}
For $m=n$ we abbreviate $\succeq_n$ to $\succeq$. A finite positive measure $\mu$ is \emph{maximal} if
\[
\nu\succeq_m\mu,\qquad \nu(\Omega)=\mu(\Omega)
\quad\Longrightarrow\quad \nu=\mu.
\]
A function $u\in\F_m(\Omega)$ is \emph{minimal} if, for every $v\in\F_m(\Omega)$,
\[
v\le u,\qquad \int_\Omega H_m(v)=\int_\Omega H_m(u)
\quad\Longrightarrow\quad v=u.
\]
Here maximality refers to the order on measures; it is different from the usual notion of an $m$-maximal function used in~\cite{Ng19}. Nilsson and Wikstr\"om~\cite{NW25} studied related plurisubharmonic orderings through Edwards' duality theorem, using test classes that impose equality of total mass. In the present paper, this equality is included explicitly in both definitions.

By~\cite[Proposition~5]{Ng18}, maximality of $H_m(u)$ implies minimality of $u\in\F_m$. The converse holds in one complex dimension~\cite[Proposition~4.11]{Be17}. At the end of~\cite{Ng18}, the second author asked whether it remains true in higher dimension. Our first result gives a negative answer for every nonlinear order $2\le m\le n$, including the Monge--Amp\`ere case.

Let $B$ be the unit ball in $\C^n$, with coordinates $z=(\zeta,w)\in\C\times\C^{n-1}$, and put $\mathcal H=\{\operatorname{Im}\zeta=0\}\cap B$.

Let $q=n-1$, let $b_q=\pi^q/q!$ be the volume of the unit ball in $\C^q$, and define the probability measures
\begin{equation}\label{hessian-family-measures}
d\lambda_0(x)=2\one_{[-1/4,1/4]}(x)\,dx,\qquad
d\tau_r(w)=\frac{\one_{\{|w|<r\}}}{b_qr^{2q}}\,dV(w),\qquad
\mu_r=\lambda_0\otimes\tau_r,
\end{equation}
where $0<r<1/2$ and $\lambda_0$ is regarded as a measure on the real axis in the $\zeta$ plane. Each $\mu_r$ is compactly supported in $B\cap\mathcal H$ and has mass one.

\begin{theorem}\label{thm:hessian-counterexamples}
Let $2\le m\le n$. For every $0<r<1/2$, there is a unique
\[
u_r\in\E_{0,m}(B)\cap C(\overline B),\qquad
H_m(u_r)=\mu_r.
\]
The function $u_r$ is H\"older continuous on $\overline B$ and minimal in $\F_m(B)$. If $0<s<r<1/2$, then
\begin{equation}\label{hessian-chain}
H_m(u_s)=\mu_s\succeq_m\mu_r=H_m(u_r),\qquad
\mu_s\ne\mu_r.
\end{equation}
Consequently, every $u_r$ is a bounded continuous minimal function whose Hessian measure is not maximal. The competing measures in~\eqref{hessian-chain} also arise from bounded continuous minimal functions. Distinct functions in the family are incomparable in pointwise order.
\end{theorem}

The proof uses a minimality criterion for continuous potentials whose Hessian measures have compact support in $\mathcal H$. A rank-one test function turns the order inequality into equality of potentials on almost every complex hyperplane. An explicit Lipschitz subsolution supplies the potentials in Theorem~\ref{thm:hessian-counterexamples}. In~\cite[Theorem~10]{Ng18}, minimality was established when the Hessian measure is carried by an $m$-polar set. The measures in our construction instead vanish on all $m$-polar sets, as do the competing measures in~\eqref{hessian-chain}. For $m=1$, minimality and maximality are equivalent on a ball; see Proposition~\ref{prop:linear}.

We also give a second construction for $m=n$. Its potentials are explicit, which allows us to study uniform variation, $L^p$ limits, and all possible dominating measures. The main tool for this construction is partial pluricomplex energy. For hyperconvex domains $D\subset\C$ and $G\subset\C^q$, and $U\in\F(D\times G)$, define at every $z\in D$
\begin{equation}\label{partial-def-intro}
P_U(z)=\int_{D\times G}g_D(z,\zeta)(\ddc U(\zeta,w))^{q+1},
\end{equation}
where $g_D$ is the negative Green function with $\ddc_z g_D(z,\zeta)=\delta_\zeta$. Theorem~\ref{thm:partial} identifies this function with the energy of the slice wherever it is finite.

\begin{theorem}\label{thm:criterion-intro}
Let $D\subset\C$ and $G\subset\C^q$, $q\ge1$, be hyperconvex domains. If $U\in\F(D\times G)$ and $P_U$ is minimal in $\F(D)$, then $U$ is minimal in $\F(D\times G)$. In particular, $U$ is minimal whenever
\[
\pi_*\bigl((\ddc U)^{q+1}\bigr)
\]
is a maximal measure on $D$, where $\pi(z,w)=z$.
\end{theorem}

We next describe the explicit family. Fix $n\ge2$, and put $q=n-1$ and $c=\log2$. Set
\begin{equation}\label{ellipse}
E=\left\{x+iy\in\C:
\frac{x^2}{(5/4)^2}+\frac{y^2}{(3/4)^2}<1\right\},
\qquad \Omega=E\times\D^q,
\end{equation}
where $\D$ denotes the unit disc. On $K=[-1,1]$, let $\lambda$ be the arcsine probability measure
\begin{equation}\label{lambda}
d\lambda(x)=\frac{\one_{(-1,1)}(x)}{\pi\sqrt{1-x^2}}\,dx,
\end{equation}
viewed as a measure on $E$. Its logarithmic potential is
\begin{equation}\label{h-def}
h(z)=\int_{-1}^{1}\log|z-x|\,d\lambda(x).
\end{equation}
We shall verify that $h\in\E_0(E)\cap C(\overline E)$, $\ddc h=\lambda$, $h=-c$ on $K$, and $h=0$ on $\partial E$.

For $t>0$, define
\begin{equation}\label{Ut}
U_t(z,w)=\max\left\{
t^{q/n}h(z),\ t^{-1/n}\log|w_2|,\ldots,
t^{-1/n}\log|w_n|
\right\}.
\end{equation}
We use $\log0=-\infty$. For $0<r<1$, let $\sigma_r$ be normalized arclength measure on $\{|\zeta|=r\}$.

\begin{theorem}\label{thm:family}
For every $n\ge2$, the functions in~\eqref{Ut} satisfy the following properties.
\begin{enumerate}
\item For every $t>0$, $U_t\in\E_0(\Omega)\cap C(\overline\Omega)$, $U_t=0$ on $\partial\Omega$, and
\[
-c\,t^{q/n}\le U_t<0\quad\text{on }\Omega.
\]
\item Each $U_t$ is minimal in the whole class $\F(\Omega)$, and
\begin{equation}\label{mut}
\mu_t:=(\ddc U_t)^n
=\lambda\otimes\sigma_{2^{-t}}^{\otimes q},\qquad
\mu_t(\Omega)=1.
\end{equation}
\item If $0<t_1<t_2$, then
\[
\mu_{t_2}\succeq\mu_{t_1},\qquad \mu_{t_2}\ne\mu_{t_1}.
\]
Consequently, no $\mu_t$ is maximal, even when competing measures are required to be Monge--Amp\`ere measures of bounded continuous minimal functions in $\E_0(\Omega)$.
\item Distinct functions in the family are not comparable in pointwise order. Moreover, $t\mapsto U_t$ is continuous from $(0,\infty)$ into $C(\overline\Omega)$ equipped with the uniform norm.
\end{enumerate}
\end{theorem}

The theorem distinguishes the two order structures even within this family of continuous minimal functions. In particular, the failure of maximality can be witnessed by functions arbitrarily close in the uniform norm. No measure charging a pluripolar set is needed as a competitor.

Section~\ref{sec:hessian} proves Theorem~\ref{thm:hessian-counterexamples} and treats the linear endpoint. Section~\ref{sec:prelim} develops the partial-energy criterion. Sections~\ref{sec:planar} and~\ref{sec:family} give the explicit Monge--Amp\`ere construction, and Section~\ref{sec:limits} studies its limits and dominating measures.

\section{Counterexamples for every nonlinear Hessian order}\label{sec:hessian}

We work on the unit ball $B\subset\C^n$ and write
\[
\beta=\beta_\zeta+\beta_w,\qquad
\beta_\zeta=\ddc|\zeta|^2,\qquad \beta_w=\ddc|w|^2.
\]
For a bounded $m$-hyperconvex domain $G\subset\C^d$, let $\E_{0,m}(G)$ consist of the bounded negative $m$-subharmonic functions with zero boundary values and finite total Hessian mass. The class $\F_m(G)$ consists of decreasing limits $v_j\searrow v$, with $v_j\in\E_{0,m}(G)$ and $\sup_j\int_G H_m(v_j)<\infty$. Replacing this bound by
\[
\sup_j\int_G(-v_j)H_m(v_j)<\infty
\]
defines $\E_{1,m}(G)$. Here $H_m$ is computed with the Euclidean form on $\C^d$. We use the approximation and mixed-product calculus of~\cite{Lu15,Ng18}.

The proof of Theorem~\ref{thm:hessian-counterexamples} uses slices of complex codimension one. It applies throughout $2\le m\le n$ and is independent of partial pluricomplex energy.

\subsection{Slicing and a minimality criterion}

For $|\zeta|<1$, put
\[
B_\zeta=\{w\in\C^{n-1}:|w|^2<1-|\zeta|^2\},\qquad
v_\zeta(w)=v(\zeta,w).
\]
For $k=m-1$, the Hessian operator on this slice is
\[
H^{w}_{m-1}(a)=(\ddc_w a)^{m-1}\wedge\beta_w^{n-m}.
\]

\begin{lemma}\label{lem:hessian-slices}
Let $2\le m\le n$ and $v\in\F_m(B)$. Then $v_\zeta\in\F_{m-1}(B_\zeta)$ for almost every $\zeta\in\D$. Moreover, the positive measure
\[
T_v=\beta_\zeta\wedge(\ddc v)^{m-1}\wedge\beta^{n-m}
\]
has finite mass and satisfies, for every nonnegative Borel function $f$ on $B$,
\begin{equation}\label{hessian-disintegration}
\int_B f\,dT_v
=\int_\D\left(\int_{B_\zeta}f(\zeta,w)\,
H^w_{m-1}(v_\zeta)\right)\beta_\zeta.
\end{equation}
Here the right-hand side is understood on the full-measure set of admissible slices.
\end{lemma}

\begin{proof}
We first justify restriction to a complex hyperplane. If $a$ is smooth and $m$-subharmonic, the mixed positivity inequality~\cite[Proposition~2.1]{Bl05} gives, for $1\le j\le m-1$,
\[
0\le\beta_\zeta\wedge(\ddc a)^j\wedge\beta^{n-1-j}
=\beta_\zeta\wedge(\ddc_w a)^j\wedge\beta_w^{n-1-j}.
\]
Indeed, wedging with $\beta_\zeta$ removes every other term containing a $\zeta$ differential; the coefficient in this identity is one. The right-hand side gives precisely the positivity conditions for the restriction to be $(m-1)$-subharmonic. Convolution on relatively compact subsets, followed by local uniform convergence, proves the same assertion for continuous $a$.

By~\cite[Theorem~3.1]{Lu15}, choose $v_j\in\E_{0,m}(B)\cap C(\overline B)$ decreasing to $v$, and put $\rho(z)=|z|^2-1$. The cone property of $\F_m$ gives $v+\rho\in\F_m(B)$. Positivity of the mixed terms and monotonicity of total mass yield the uniform estimate
\begin{align}\label{slice-mass-bound}
\int_B\beta_\zeta\wedge(\ddc v_j)^{m-1}\wedge\beta^{n-m}
&\le\int_B(\ddc v_j)^{m-1}\wedge\beta^{n-m+1}\notag\\
&\le\frac1m\int_B H_m(v_j+\rho)
\le\frac1m\int_B H_m(v+\rho)<\infty.
\end{align}
The first inequality follows because $\beta-\beta_\zeta$ is semipositive. For the cone property see~\cite[Proposition~1]{Ng18}; mass monotonicity is~\cite[Theorem~3.22]{Lu15}.

For a fixed $j$, we prove the slice formula first against a continuous test function supported in a relatively compact product subset of $B$. Smooth convolution approximants of $v_j$ converge locally uniformly, and the preceding differential-form identity gives the formula for each approximant by Fubini's theorem. On a slightly smaller product, the local Chern--Levine--Nirenberg estimates for Hessian currents, as in~\cite[Section~2]{Bl05}, give a bound for the slice masses independent of the regularization parameter and of $\zeta$. Thus local Hessian convergence and dominated convergence prove the formula for $v_j$. Every compact subset of $B$ is covered by finitely many such product neighborhoods. A partition of unity therefore gives the formula for every test function in $C_c(B)$, and equality of measures gives it for every nonnegative Borel function.

In particular, the measurable functions
\[
M_j(\zeta)=\int_{B_\zeta}H^w_{m-1}((v_j)_\zeta)
\]
have uniformly bounded integrals by~\eqref{slice-mass-bound}. Remove the union of the null sets $\{M_j=\infty\}$ over all $j$. For every remaining $\zeta$, the function $(v_j)_\zeta$ is bounded and continuous on $\overline{B_\zeta}$ and vanishes on its boundary, since $\{\zeta\}\times\partial B_\zeta\subset\partial B$. Hence $(v_j)_\zeta\in\E_{0,m-1}(B_\zeta)$. As $(v_{j+1})_\zeta\le(v_j)_\zeta$, mass monotonicity on $B_\zeta$ gives $M_j(\zeta)\le M_{j+1}(\zeta)$. Monotone convergence now shows that $M:=\sup_j M_j$ is integrable. Also, local integrability of $v$ and Fubini's theorem imply that $v_\zeta$ is locally integrable outside a null set. Removing this set and $\{M=\infty\}$ gives a common full-measure set $Z\subset\D$. For every $\zeta\in Z$, the sequence $(v_j)_\zeta$ proves that $v_\zeta\in\F_{m-1}(B_\zeta)$.

Let $f\in C_c(B)$. On each slice with $\zeta\in Z$, decreasing Hessian convergence gives convergence of the inner integral in~\eqref{hessian-disintegration}; its absolute value is bounded by $\|f\|_\infty M(\zeta)$. Dominated convergence therefore gives convergence of the right-hand side. To identify the left-hand side, fix a compact neighborhood $L$ of $\supp f$ in $B$ and choose $A$ large enough that
\[
\theta=\max\{|\zeta|^2-2,A\rho\}\in\E_{0,m}(B)
\]
equals $|\zeta|^2-2$ near $L$. There $\ddc\theta=\beta_\zeta$, so the mixed convergence theorem~\cite[Theorem~3.11]{Lu15}, with the fixed factor $\theta$, gives $T_{v_j}\rightharpoonup T_v$ locally. This identifies both limits. The resulting equality of measures proves~\eqref{hessian-disintegration} for all nonnegative Borel $f$ and gives $T_v(B)\le\int_\D M\,\beta_\zeta<\infty$.
\end{proof}

\begin{lemma}\label{lem:ordered-zero}
Let $G\subset\C^d$ be a bounded $k$-hyperconvex domain, $1\le k\le d$, and let $b\le a$. Assume either $a,b\in\E_{1,k}(G)$, or $a\in\E_{0,k}(G)$ and $b\in\F_k(G)$. In the first case, define $a-b=0$ on $\{a=b=-\infty\}$; this convention does not affect integration against $H_k(b)$, since energy measures do not charge $k$-polar sets. If
\begin{equation}\label{ordered-zero}
\int_G(a-b)H_k(b)=0,
\end{equation}
then $a=b$.
\end{lemma}

\begin{proof}
In the second case, boundedness of $a$ gives
\[
\int_G(-b)H_k(b)=\int_G(-a)H_k(b)<\infty.
\]
For continuous $\E_{0,k}$ approximants $b_j\searrow b$, Proposition~2 of~\cite{Ng18} yields
\[
\int_G(-b_j)H_k(b_j)\le\int_G(-b_j)H_k(b)
\le\int_G(-b)H_k(b).
\]
Thus $b\in\E_{1,k}(G)$. Also $a\in\E_{0,k}(G)\subset\E_{1,k}(G)$, so both cases reduce to $a,b\in\E_{1,k}(G)$.

Equation~\eqref{ordered-zero} implies $H_k(b)(\{b<a\})=0$. Fix $L\Subset G$, choose a negative $\rho_0\in\E_{0,k}(G)\cap C(\overline G)$, and take $R>\sup_G|z|$. For sufficiently large $A$,
\[
\rho=\max\{A\rho_0,|z|^2-R^2\}\in\E_{0,k}(G)
\]
equals $|z|^2-R^2$ near $L$. For $\varepsilon>0$, the cone property gives $a+\varepsilon\rho\in\E_{1,k}(G)$. With $\beta_G=\ddc|z|^2$, positivity of mixed Hessians and the comparison principle~\cite[Theorem~3.24]{Lu15} give
\begin{align*}
\varepsilon^k\int_{\{b<a+\varepsilon\rho\}}
(\ddc\rho)^k\wedge\beta_G^{d-k}
&\le\int_{\{b<a+\varepsilon\rho\}}H_k(a+\varepsilon\rho)\\
&\le\int_{\{b<a+\varepsilon\rho\}}H_k(b)=0.
\end{align*}
The last equality holds because $\rho<0$, so the integration set is contained in $\{b<a\}$. As $\varepsilon\downarrow0$, these sets increase to $\{b<a\}$. Since $(\ddc\rho)^k\wedge\beta_G^{d-k}=\beta_G^d$ near $L$, we obtain $\vol(L\cap\{b<a\})=0$. Thus $a=b$ almost everywhere on $G$. Since subharmonic functions are determined by their almost-everywhere equivalence classes, the equality holds everywhere, including their common $-\infty$ locus.
\end{proof}

\begin{proposition}\label{prop:hyperplane}
Let $2\le m\le n$ and $u\in\E_{0,m}(B)\cap C(\overline B)$. Assume that $H_m(u)$ is supported on a compact subset of
\[
\mathcal H=\{(\zeta,w)\in B:\operatorname{Im}\zeta=0\}.
\]
Then $u$ is minimal in the whole class $\F_m(B)$.
\end{proposition}

\begin{proof}
Let $v\in\F_m(B)$ satisfy $v\le u$ and
\[
\int_B H_m(v)=\int_B H_m(u)=M.
\]
Set $\chi(\zeta,w)=(\operatorname{Im}\zeta)^2-1$ and, for $A>0$, define
\[
\psi_A=\max\{\chi,A\rho\},\qquad \rho(z)=|z|^2-1.
\]
Then $\psi_A\in\E_{0,m}(B)\cap C(\overline B)$ and $-1\le\psi_A<0$. For every sufficiently large $A$, compactness of $\supp H_m(u)\subset\mathcal H$ gives $\psi_A=-1$ on that support. Since $v\le u$, Proposition~2 of~\cite{Ng18} gives
\[
M=\int_B(-\psi_A)H_m(u)
\le\int_B(-\psi_A)H_m(v)\le M.
\]
All three quantities are therefore equal. We next justify integration by parts with the possibly unbounded function $v$. For $0\le j\le m-1$, set
\[
R_j=(\ddc v)^j\wedge(\ddc u)^{m-1-j}\wedge\beta^{n-m}.
\]
The measures $\ddc u\wedge R_j$ and $\ddc v\wedge R_j$ have finite mass: each occurs, up to a positive binomial coefficient, in the expansion of $H_m(u+v)$, and $u+v\in\F_m(B)$. Since $\psi_A$ is bounded, integration by parts~\cite[Theorem~3.16]{Lu15} gives
\begin{align*}
\int_B(-v)\ddc\psi_A\wedge R_j
&=\int_B(-\psi_A)\ddc v\wedge R_j<\infty,\\
\int_B(-u)\ddc\psi_A\wedge R_j
&=\int_B(-\psi_A)\ddc u\wedge R_j<\infty.
\end{align*}
The cited theorem applies in $\F_m$ as soon as one side is finite; its proof uses decreasing $\E_{0,m}$ approximants. Thus both identities hold without any boundedness assumption on $v$, and their difference is well defined. Summing these differences gives the telescoping identity
\begin{align}\label{saturation-identity}
0&=\int_B(-\psi_A)\bigl(H_m(v)-H_m(u)\bigr)\notag\\
&=\sum_{j=0}^{m-1}\int_B(u-v)\ddc\psi_A\wedge
(\ddc v)^j\wedge(\ddc u)^{m-1-j}\wedge\beta^{n-m}.
\end{align}
Each summand is finite by the preceding identities and nonnegative because $v\le u$. Hence each is zero.

For a fixed $A$, put $O_A=\{\chi>A\rho\}$. On this open set $\psi_A=\chi$ and $\ddc\chi=\tfrac12\beta_\zeta$. To use this identity in the mixed measure, take bounded approximants $v_l\searrow v$. Locality gives the identity with $v_l$ on $O_A$, and local mixed convergence as in Lemma~\ref{lem:hessian-slices} gives
\[
\one_{O_A}\ddc\psi_A\wedge(\ddc v)^{m-1}\wedge\beta^{n-m}
=\tfrac12\one_{O_A}T_v.
\]
This is an equality of measures, so it may be integrated against the nonnegative Borel function $u-v$. The term $j=m-1$ in~\eqref{saturation-identity} therefore gives $\int_{O_A}(u-v)\,dT_v=0$. Finally, take integers $A\to\infty$. Since $\rho<0$ on $B$, the sets $O_A$ increase to $B$, and monotone convergence yields
\begin{equation}\label{hyperplane-zero}
\int_B(u-v)\,\beta_\zeta\wedge(\ddc v)^{m-1}\wedge\beta^{n-m}=0.
\end{equation}
Apply Lemma~\ref{lem:hessian-slices} to $u$ and $v$ and intersect their full-measure sets of admissible slices. Tonelli's theorem and~\eqref{hyperplane-zero} then give, for almost every $\zeta$ in this intersection,
\[
\int_{B_\zeta}(u_\zeta-v_\zeta)H^w_{m-1}(v_\zeta)=0.
\]
Here $u_\zeta$ is bounded and continuous up to $\partial B_\zeta$, with zero boundary values and finite mass; hence $u_\zeta\in\E_{0,m-1}(B_\zeta)$, while $v_\zeta\in\F_{m-1}(B_\zeta)$. Lemma~\ref{lem:ordered-zero} gives $v_\zeta=u_\zeta$ on each such slice. Fubini's theorem implies $v=u$ almost everywhere on $B$. Since subharmonic functions are determined by their almost-everywhere equivalence classes, $v=u$ everywhere.
\end{proof}

\begin{remark}
The function $\chi=(\operatorname{Im}\zeta)^2-1$ has $\ddc\chi=\frac12\beta_\zeta$. For $m=1$, it is strictly subharmonic, and both $\chi\pm\varepsilon f$ remain negative and subharmonic for every real $f\in C_c^\infty(B)$ and sufficiently small $\varepsilon>0$. If a finite measure is carried by $\mathcal H$, saturation of its integral against $\chi$ therefore forces every dominating measure of equal mass to coincide with it. The extension to these negative test functions is justified by~\cite[Remark~2(1)]{Ng18}.

For $m\ge2$, the rank-one form $\ddc\chi$ lies on the boundary of the cone of $m$-positive forms: its second elementary symmetric function vanishes. Arbitrary two-sided perturbations are no longer admissible. Nevertheless, for ordered potentials the saturation identity~\eqref{saturation-identity} still gives equality on slices, and hence minimality. The measures in Theorem~\ref{thm:hessian-counterexamples} show that this equality of potentials does not force maximality of the measure.
\end{remark}

\subsection{An explicit chain of measures}

We prove Theorem~\ref{thm:hessian-counterexamples} for the measures defined in~\eqref{hessian-family-measures}.

\begin{proof}[Proof of Theorem~\ref{thm:hessian-counterexamples}]
We first prove existence and regularity. Consider the Lipschitz psh function
\[
p(z)=|\operatorname{Im}\zeta|+|z|^2-3,
\qquad b(z)=\max\{p(z),8(|z|^2-1)\}.
\]
It is negative on $B$, continuous on $\overline B$, and zero on $\partial B$. Since $b\ge8\rho$, mass monotonicity shows that $b\in\E_{0,m}(B)$. On a neighborhood of
\[
\{(x,w):|x|\le1/4,\ |w|\le1/2\}\subset\mathcal H
\]
we have $b=p$. Indeed, on this compact set $|z|^2\le5/16$, and $p-8\rho=5-7|z|^2>0$.

Put $S=\ddc|\operatorname{Im}\zeta|$. This positive current depends only on $\zeta$ and has rank one, so local regularization gives $S^2=0$. Hence, near $\supp\mu_r$,
\begin{equation}\label{subsolution-hessian}
H_m(b)=\beta^n+mS\wedge\beta^{n-1}.
\end{equation}
The second term is a positive constant times $dx\,dV(w)$ on $\mathcal H$. With normalization~\eqref{normalization}, it is exactly
\[
mS\wedge\beta^{n-1}
=\frac{m\,2^{n-1}(n-1)!}{\pi^n}\,
dx\otimes\delta_0(d\operatorname{Im}\zeta)\otimes dV(w).
\]
Comparison with~\eqref{hessian-family-measures} therefore gives
\begin{equation}\label{explicit-subsolution}
\mu_r\le H_m(a_rb),\qquad
a_r^m=\frac{2^{2-n}\pi}{m r^{2n-2}}.
\end{equation}
This proves the required subsolution condition directly. In particular, $\mu_r$ does not charge $m$-polar sets.

By~\cite[Lemma~5.1]{Lu15}, there is a unique $u_r\in\E_{0,m}(B)$ with $H_m(u_r)=\mu_r$. We verify the hypotheses of Theorem~B$'$ of~\cite{BZ21}. The smooth defining function $\rho$ has $\ddc\rho=\beta>0$, so $B$ is strongly $m$-pseudoconvex. The measure $\mu_r$ is positive and has mass one. The subsolution $a_rb$ is Lipschitz, hence belongs to $\E_{0,m}(B)\cap C^{0,1/2}(\overline B)$, has zero boundary values, and satisfies~\eqref{explicit-subsolution}. Finally, the prescribed boundary datum $0$ belongs to $C^{1,1}(\partial B)$. The cited theorem therefore gives a H\"older continuous solution on $\overline B$. By comparison with the subsolution it is bounded below by $a_rb$, so it belongs to $\E_{0,m}(B)$ and equals $u_r$ by uniqueness. We use this corrected version of~\cite{BZ20} without specifying an exponent. Proposition~\ref{prop:hyperplane} proves minimality.

We next prove the order relation. Let $\varphi\in\E_{0,m}(B)\cap C(B)$. For each $x\in[-1/4,1/4]$, its restriction $w\mapsto\varphi(x,w)$ is $(m-1)$-subharmonic, hence subharmonic, on $B_x$. Ball averages of a subharmonic function are nondecreasing with the radius. Consequently,
\[
\int\varphi(x,w)\,d\tau_s(w)
\le\int\varphi(x,w)\,d\tau_r(w)
\quad(0<s<r<1/2).
\]
Integration against $\lambda_0$ gives $\mu_s\succeq_m\mu_r$. The measures are distinct, since $\mu_s$ gives zero mass to $\{s<|w|<r\}$ whereas $\mu_r$ gives it positive mass. 

Both Hessian masses equal one. If two distinct potentials in the family were pointwise comparable, minimality of the larger one would force equality. This proves the last assertion.
\end{proof}

\begin{remark}
For $m<n$, restriction of an $m$-subharmonic function to an arbitrary complex line need not be subharmonic. For example,
\[
Q(z)=-a|z_1|^2+\sum_{j=2}^n|z_j|^2,
\qquad 0<a<\frac{n-m}{m},
\]
is strictly $m$-subharmonic: for $1\le k\le m$, its $k$-th elementary symmetric function is
\[
\binom{n-1}{k}-a\binom{n-1}{k-1}>0.
\]
Its restriction to the $z_1$ axis is strictly superharmonic. This explains the use of complex hyperplanes, with a loss of one Hessian order, in Lemma~\ref{lem:hessian-slices}.
\end{remark}

\subsection{The linear endpoint}

\begin{proposition}\label{prop:linear}
On the unit ball $B\subset\C^n$, a function $u\in\F_1(B)$ is minimal if and only if $H_1(u)$ is maximal for the subharmonic ordering. Thus there is no counterexample for $m=1$ on $B$.
\end{proposition}

\begin{proof}
The implication from maximality to minimality is~\cite[Proposition~5]{Ng18}. Conversely, write $\mu=H_1(u)$ and suppose that $\nu\succeq_1\mu$ with $\nu(B)=\mu(B)$. In the linear case, $\F_1(B)$ is the class of negative Green potentials of finite positive measures. Let $g_B(x,y)\le0$ be the symmetric Green kernel normalized by $H_1(g_B(\cdot,y))=\delta_y$, and set
\[
v(x)=\int_B g_B(x,y)\,d\nu(y).
\]
Classical Green potential theory gives $v\in\F_1(B)$, $H_1(v)=\nu$, and $u(x)=\int_B g_B(x,y)\,d\mu(y)$. These statements also hold for measures charging polar sets, since $H_1$ is a positive constant multiple of the Laplacian. For fixed $x\in B$, the truncated kernels $g_{x,j}=\max\{g_B(x,\cdot),-j\}$ belong to $\E_{0,1}(B)\cap C(\overline B)$ and decrease to $g_B(x,\cdot)$. Testing the order against $g_{x,j}$ and using monotone convergence for $-g_{x,j}$ gives $v(x)\le u(x)$, with extended values allowed. Minimality of $u$ yields $v=u$, and hence $\nu=\mu$.
\end{proof}

\section{Partial energy and the product criterion}\label{sec:prelim}

On a hyperconvex domain $H\subset\C^d$, we write $\E_0=\E_{0,d}$, $\F=\F_d$ and $\E_1=\E_{1,d}$. The local Cegrell class $\E(H)$ contains $\F(H)$, $\E_1(H)$ and all locally bounded negative psh functions. Monge--Amp\`ere measures and mixed products are understood in the sense of Bedford--Taylor and Cegrell~\cite{BT82,Ce98,Ce04}.

We shall use integration by parts and the comparison principle in $\E_1$. Mixed energy integrals are finite by Cegrell's energy inequality. For example, for $v_0,\ldots,v_d\in\E_1(H)$,
\[
\int_H(-v_0)\ddc v_1\wedge\cdots\wedge\ddc v_d
\le\prod_{j=0}^{d}
\left(\int_H(-v_j)(\ddc v_j)^d\right)^{1/(d+1)}.
\]
For the mixed energy inequality see~\cite{Ce98} and~\cite[Theorem~1.1, $p=1$]{CN17}; for the energy calculus see~\cite{ACC12}. Integration by parts and comparison in $\E_1$ also follow from~\cite[Theorems~3.21 and~3.24]{Lu15} with $m=d$.

\begin{lemma}\label{lem:strict-energy}
Let $a,b\in\E_1(H)$ and $a\le b$. Set
\[
I_H(a)=\int_H a(\ddc a)^d.
\]
Then $I_H(a)\le I_H(b)$, and equality holds if and only if $a=b$.
\end{lemma}

\begin{proof}
Set $S=\ddc a$ and $T=\ddc b$. All mixed energy integrals below are finite. For $1\le j\le d$, integration by parts in $\E_1$ gives
\[
\int_H bS^j\wedge T^{d-j}
=\int_H aS^{j-1}\wedge T^{d-j+1}.
\]
Therefore, on expanding $b-a$ in each integral, every term on the right-hand side below cancels with its neighbor except $\int_H bT^d$ and $-\int_H aS^d$:
\begin{equation}\label{energy-difference}
I_H(b)-I_H(a)
=\sum_{j=0}^{d}\int_H(b-a)S^j\wedge T^{d-j}.
\end{equation}
The mixed measures do not charge pluripolar sets, so the common $-\infty$ locus of $a$ and $b$ does not affect these integrals. All summands are nonnegative. If their sum is zero, the term $j=d$ gives $\int_H(b-a)(\ddc a)^d=0$. The finite-energy case of Lemma~\ref{lem:ordered-zero}, with the roles of $a$ and $b$ interchanged, gives $a=b$.
\end{proof}

We next recall the partial energy theorem. Let $g_D$ denote the Green function of a planar hyperconvex domain $D$, normalized by $\ddc_z g_D(z,\zeta)=\delta_\zeta$. If $\mu$ is a finite positive measure on $D\times G$, its pushforward by $\pi(z,w)=z$ is the finite measure defined by
\[
(\pi_*\mu)(A)=\mu(A\times G),\qquad A\subset D\text{ Borel}.
\]
Equivalently, $\int_D f\,d(\pi_*\mu)=\int_{D\times G}f\circ\pi\,d\mu$ for every nonnegative Borel function $f$. This definition does not require $\pi$ to be proper.

\begin{theorem}[\cite{ACKPZ09}, Theorem~3.1]\label{thm:partial}
Let $D\subset\C$ and $G\subset\C^q$ be hyperconvex domains, and let $U\in\F(D\times G)$. The Green potential $P_U$ defined at every point by~\eqref{partial-def-intro} belongs to $\F(D)$ and satisfies
\begin{equation}\label{partial-pushforward}
\ddc P_U=\pi_*\bigl((\ddc U)^{q+1}\bigr),\qquad
\int_D\ddc P_U=\int_{D\times G}(\ddc U)^{q+1}.
\end{equation}
For every $z$ with $P_U(z)>-\infty$, the slice $U(z,\cdot)$ belongs to $\E_1(G)$ and
\begin{equation}\label{green-partial}
P_U(z)=\int_G U(z,w)\bigl(\ddc_w U(z,w)\bigr)^q.
\end{equation}
The exceptional set $\{P_U=-\infty\}$ is polar. Thus the slice-energy formula agrees with the everywhere-defined Green potential outside a polar set.
\end{theorem}

The slice assertion and the membership $P_U\in\F(D)$ are exactly~\cite[Theorem~3.1]{ACKPZ09}, with the two factors interchanged. Its hypotheses require bounded hyperconvex factors and $U\in\F(D\times G)$, but no boundary smoothness. Its normalization agrees with~\eqref{normalization}. Writing $\mu=(\ddc U)^{q+1}$, formula~\eqref{partial-def-intro} is the Green potential of $\pi_*\mu$. Consequently, for every nonnegative $f\in C_c^\infty(D)$,
\[
\int_D f\,\ddc P_U
=\int_{D\times G}f\circ\pi\,d\mu,
\]
which proves the pushforward identity in~\eqref{partial-pushforward}. The equality of masses follows by evaluating $\pi_*\mu$ on $D$.

The following special case of the product formula will be used both on $\Omega$ and on its slices.

\begin{lemma}\label{lem:product}
Let $D_j\subset\C^{d_j}$, $1\le j\le k$, be hyperconvex domains, and put $N=\sum_{j=1}^k d_j$. Suppose $f_j\in\E_0(D_j)\cap C(\overline D_j)$, $f_j\ge-\ell$ for some $\ell>0$, and $(\ddc f_j)^{d_j}$ is carried by $\{f_j=-\ell\}$. Then
\[
F(z_1,\ldots,z_k)=\max_{1\le j\le k}f_j(z_j)
\]
belongs to $\E_0(D_1\times\cdots\times D_k)\cap C(\overline{D_1\times\cdots\times D_k})$, and
\begin{equation}\label{product-measure}
(\ddc F)^N=\bigotimes_{j=1}^k(\ddc f_j)^{d_j}.
\end{equation}
\end{lemma}

\begin{proof}
For two factors, Corollary~2.1 of~\cite{ACP10} gives finite total mass for $F$. Since $F$ is bounded, continuous and zero on the boundary, it belongs to $\E_0$. The pullbacks of $f_1,f_2$ are locally bounded, and
\[
(\ddc f_1)^{d_1}\wedge(\ddc f_2)^{d_2}
=(\ddc f_1)^{d_1}\otimes(\ddc f_2)^{d_2}
\]
charges no pluripolar set and is carried by $\{f_1=f_2=-\ell\}$. Corollary~4.3 of~\cite{ACP10} therefore identifies it with $(\ddc F)^{d_1+d_2}$, with coefficient one under our normalization. Induction gives~\eqref{product-measure}, since the maximum of the first factors again has its measure carried by $\{-\ell\}$.
\end{proof}

\subsection{Minimality from the planar projection}

We now prove the criterion stated in the introduction. The point is that partial energy reduces a comparison on $D\times G$ to a one-dimensional comparison on the base $D$ and then to strict energy on almost every slice.

\begin{proof}[Proof of Theorem~\ref{thm:criterion-intro}]
Let $V\in\F(D\times G)$ satisfy $V\le U$ and have the same total Monge--Amp\`ere mass. Theorem~\ref{thm:partial} gives $P_U,P_V\in\F(D)$ and
\begin{equation}\label{equal-projected-masses}
\int_D\ddc P_V=\int_D\ddc P_U.
\end{equation}
Let
\[
N=\{P_U=-\infty\}\cup\{P_V=-\infty\}.
\]
This is a polar set. For every $z\in D\setminus N$, Theorem~\ref{thm:partial} gives both slice-energy identities and puts both slices in $\E_1(G)$. Since $V(z,\cdot)\le U(z,\cdot)$, Lemma~\ref{lem:strict-energy} yields $P_V(z)\le P_U(z)$. Thus the inequality holds almost everywhere on $D$. To obtain it at an arbitrary point $z$, average over discs centered at $z$ and let their radii tend to zero. These averages converge to the values of the subharmonic functions, with $-\infty$ allowed. Hence $P_V\le P_U$ everywhere.

Minimality of $P_U$ and~\eqref{equal-projected-masses} now give $P_V=P_U$. For each $z\in D\setminus N$, the two slice energies are therefore finite and equal. The strict part of Lemma~\ref{lem:strict-energy} implies $V(z,\cdot)=U(z,\cdot)$ on $G$. Since $N$ has zero planar Lebesgue measure, Fubini's theorem gives $V=U$ almost everywhere on $D\times G$. Plurisubharmonic functions are determined by their almost-everywhere equivalence classes, so $V=U$ everywhere.

Finally, if $\pi_*((\ddc U)^{q+1})$ is maximal, then $P_U$ is minimal by~\eqref{partial-pushforward} and~\cite[Proposition~5]{Ng18}. The first assertion applies.
\end{proof}

\section{The one-dimensional potential}\label{sec:planar}

We use the notation in~\eqref{ellipse}--\eqref{h-def}.

\begin{lemma}\label{lem:h}
The potential $h$ satisfies
\begin{equation}\label{h-properties}
h\in\E_0(E)\cap C(\overline E),\qquad
\ddc h=\lambda,\qquad
-c\le h<0\text{ on }E,
\end{equation}
and
\[
h=-c\text{ on }K,\qquad h=0\text{ on }\partial E.
\]
Moreover, $\lambda$ is maximal on $E$, and $h$ is minimal in $\F(E)$.
\end{lemma}

\begin{proof}
The change of variables $x=\cos\theta$ shows that $\lambda(K)=1$ and
\[
h(z)=\frac1\pi\int_0^\pi\log|z-\cos\theta|\,d\theta.
\]
For $z\notin K$, put $\zeta=z+\sqrt{z^2-1}$, where the branch is chosen so that $\sqrt{z^2-1}\sim z$ at infinity. Then $|\zeta|>1$, $z=(\zeta+\zeta^{-1})/2$, and
\[
z-\cos\theta=\frac{(\zeta-e^{i\theta})(\zeta-e^{-i\theta})}{2\zeta}.
\]
The circle mean of $\log|\zeta-e^{i\theta}|$ equals $\log|\zeta|$. It follows that
\begin{equation}\label{explicit-h}
h(z)=
\begin{cases}
\log|z+\sqrt{z^2-1}|-\log2,&z\notin K,\\
-\log2,&z\in K.
\end{cases}
\end{equation}
The value on $K$ follows from the same integrable logarithmic expression on $|\zeta|=1$. The formula is continuous across $K$, including its endpoints.

The Joukowski map sends $|\zeta|=2$ onto $\partial E$, since
\[
\frac{2e^{i\theta}+(2e^{i\theta})^{-1}}2
=\frac54\cos\theta+i\frac34\sin\theta.
\]
Thus~\eqref{explicit-h} gives $h=0$ on $\partial E$ and $-c\le h<0$ on $E$. The distributional identity $\ddc\log|z-x|=\delta_x$ gives $\ddc h=\lambda$. Hence $h\in\E_0(E)$.

To prove maximality, suppose that $\eta\succeq\lambda$ and $\eta(E)=1$. By~\cite[Remark~2(1)]{Ng18}, the order inequality extends to all negative subharmonic test functions. Consider
\[
\psi(z)=(\operatorname{Im}z)^2-1.
\]
This is negative and strictly subharmonic on $E$, $\psi\ge-1$, and $\psi=-1$ precisely on $E\cap\R$, which contains $K$. Therefore
\[
-1\le\int_E\psi\,d\eta
\le\int_E\psi\,d\lambda=-1.
\]
The equality criterion~\cite[Lemma~1]{Ng18}, applied to the strictly subharmonic test $\psi$, gives $\eta=\lambda$. Hence $\lambda$ is maximal, and~\cite[Proposition~5]{Ng18} implies minimality of $h$.
\end{proof}

\section{The family of counterexamples}\label{sec:family}

\begin{proof}[Proof of Theorem~\ref{thm:family}]
The domain $\Omega$ is hyperconvex. Indeed, the maximum of a negative continuous subharmonic exhaustion of $E$ and the functions $|w_j|^2-1$ is a bounded continuous psh exhaustion of the product.

Fix $t>0$, and write
\begin{equation}\label{coefficients}
\alpha=t^{q/n},\qquad \gamma=t^{-1/n},\qquad
\ell=c\alpha,\qquad L=ct,\qquad r=e^{-L}=2^{-t}.
\end{equation}
Then $\alpha\gamma^q=1$ and $\gamma L=\ell$. Define
\[
f_1(z)=\alpha h(z),\qquad
f_j(w_j)=\gamma\max\{\log|w_j|,-L\},\quad 2\le j\le n.
\]
All these functions are continuous on the closures of their respective factors, belong to $\E_0$ there, and have minimum $-\ell$. Furthermore,
\[
\ddc f_1=\alpha\lambda,\qquad
\ddc f_j=\gamma\sigma_r\quad(2\le j\le n).
\]
The first measure is carried by $K=\{f_1=-\ell\}$; the others are carried by $\{|w_j|=r\}\subset\{f_j=-\ell\}=\{|w_j|\le r\}$.

Since $f_1\ge-\ell$, truncating the logarithms does not change their maximum with $f_1$. Hence $U_t=\max\{f_1,\ldots,f_n\}$. Lemma~\ref{lem:product} gives
\begin{equation}\label{MA-computation}
(\ddc U_t)^n
=\alpha\gamma^q\lambda\otimes\sigma_r^{\otimes q}
=\lambda\otimes\sigma_r^{\otimes q}.
\end{equation}
It also gives $U_t\in\E_0(\Omega)\cap C(\overline\Omega)$. The boundary value is zero, and the stated bounds follow from the bounds on the factors. This proves~\textup{(i)} and the measure identity in~\textup{(ii)}.

We calculate the partial energy explicitly. For fixed $z\in E$, put $a=\alpha h(z)<0$. The slice is
\[
U_t(z,w)=\max\{a,\gamma\log|w_2|,\ldots,\gamma\log|w_n|\}.
\]
Applying Lemma~\ref{lem:product} to the truncated logarithms on this slice yields
\[
\bigl(\ddc_w U_t(z,\cdot)\bigr)^q
=\gamma^q\sigma_{\exp(a/\gamma)}^{\otimes q}.
\]
The slice equals $a$ on the support of this measure. Therefore
\begin{equation}\label{P-equals-h}
P_{U_t}(z)=a\gamma^q
=\alpha\gamma^q h(z)=h(z).
\end{equation}
Lemma~\ref{lem:h} and Theorem~\ref{thm:criterion-intro} prove that $U_t$ is minimal in $\F(\Omega)$. Equivalently, one can apply the marginal assertion of that theorem directly to~\eqref{MA-computation}.

To prove~\textup{(iii)}, take $0<t_1<t_2$ and set $r_j=2^{-t_j}$, so that $r_2<r_1$. For $\varphi\in\E_0(\Omega)\cap C(\Omega)$ and fixed $x\in K$, the function $w\mapsto\varphi(x,w)$ is bounded and psh. Its circle mean in each variable is nondecreasing with the radius. Iterating this inequality and then integrating against $\lambda$ gives
\begin{align*}
\int_\Omega\varphi\,d\mu_{t_2}
&=\int_K\int_{\D^q}\varphi(x,w)\,d\sigma_{r_2}^{\otimes q}(w)\,d\lambda(x)\\
&\le\int_K\int_{\D^q}\varphi(x,w)\,d\sigma_{r_1}^{\otimes q}(w)\,d\lambda(x)
=\int_\Omega\varphi\,d\mu_{t_1}.
\end{align*}
Thus $\mu_{t_2}\succeq\mu_{t_1}$. They are carried by $K\times\{|w_j|=r_2\text{ for all }j\}$ and $K\times\{|w_j|=r_1\text{ for all }j\}$, respectively. These sets are disjoint, so the two probability measures are distinct. Since each $U_t$ is bounded, every $\mu_t$ charges no pluripolar set.

If two distinct members were ordered pointwise, minimality of the larger member and equality of total mass would force equality of the functions. Their distinct Monge--Amp\`ere measures rule this out.

For a compact interval $J\Subset(0,\infty)$, take $M\ge c\max J$ and put $\ell_M(\zeta)=\max\{\log|\zeta|,-M\}$. Since $h\ge-c$ and $\alpha/\gamma=t$, one has
\[
U_t=\max\{\alpha(t)h,\gamma(t)\ell_M(w_2),\ldots,\gamma(t)\ell_M(w_n)\}
\quad(t\in J).
\]
Consequently,
\[
\|U_t-U_s\|_\infty
\le\max\{c|\alpha(t)-\alpha(s)|,M|\gamma(t)-\gamma(s)|\},
\qquad s,t\in J,
\]
which proves~\textup{(iv)}.
\end{proof}

\begin{corollary}\label{cor:nearby}
For every $t_0>0$ and $\varepsilon>0$, there exists a minimal function $V\in\E_0(\Omega)\cap C(\overline\Omega)$ such that
\[
\|V-U_{t_0}\|_{C(\overline\Omega)}<\varepsilon,
\qquad (\ddc V)^n\succeq(\ddc U_{t_0})^n,
\qquad (\ddc V)^n\ne(\ddc U_{t_0})^n,
\]
and both measures have total mass one.
\end{corollary}

\begin{proof}
Take $V=U_t$ with $t>t_0$ sufficiently close to $t_0$.
\end{proof}

\section{Limits and restrictions on dominating measures}\label{sec:limits}

The following proposition records the behavior of the same family as the radii shrink to zero.

\begin{proposition}\label{prop:limits}
As $t\to\infty$, one has
\begin{equation}\label{Lp-limit}
U_t\longrightarrow0\quad\text{in }L^p(\Omega,dV)
\quad\text{for every }0<p<\infty,
\end{equation}
while
\begin{equation}\label{measure-limit}
(\ddc U_t)^n\rightharpoonup
\nu_0:=\lambda\otimes\delta_0^{\otimes q}.
\end{equation}
The measure $\nu_0$ is maximal. In addition,
\begin{equation}\label{energy-divergence}
\int_\Omega(-U_t)(\ddc U_t)^n=c\,t^{q/n}\longrightarrow\infty.
\end{equation}
\end{proposition}

\begin{proof}
Set $g(w)=\max_{2\le j\le n}\log|w_j|$. Since $U_t\ge t^{-1/n}g$, we have, almost everywhere,
\[
0\le-U_t(z,w)\le t^{-1/n}(-g(w)).
\]
For every $p>0$, $g\in L^p(\D^q)$. For instance,
\[
\vol\{w\in\D^q:-g(w)>\tau\}=\pi^q e^{-2q\tau},\qquad \tau>0,
\]
and the layer-cake formula gives
\[
\int_{\D^q}(-g)^p\,dV
=\frac{\pi^q\Gamma(p+1)}{(2q)^p}<\infty.
\]
This proves~\eqref{Lp-limit}, with $\|U_t\|_{L^p}\le C_{p,n,E}t^{-1/n}$ for $p\ge1$, and the corresponding integral estimate for $0<p<1$.

Since $\sigma_{2^{-t}}\rightharpoonup\delta_0$, the explicit product formula proves~\eqref{measure-limit}. On the support of $(\ddc U_t)^n$, the value of $U_t$ is $-c\,t^{q/n}$. Its measure has mass one, so~\eqref{energy-divergence} follows.

For maximality of $\nu_0$, consider the negative strictly psh function
\[
\Psi(z,w)=(\operatorname{Im}z)^2+
\sum_{j=2}^n|w_j|^2-n.
\]
It satisfies $\Psi\ge-n$ on $\Omega$ and $\Psi=-n$ on $K\times\{0\}^q$. Suppose that $\eta\succeq\nu_0$ and $\eta(\Omega)=1$. The extension of the order to negative psh functions gives
\[
-n\le\int_\Omega\Psi\,d\eta
\le\int_\Omega\Psi\,d\nu_0=-n.
\]
Since $\Psi$ is strictly psh,~\cite[Lemma~1]{Ng18} gives $\eta=\nu_0$.
\end{proof}

\begin{proposition}\label{prop:marginal-fixed}
Fix $t>0$. If $\eta$ is a finite positive measure on $\Omega$ with $\eta\succeq\mu_t$ and $\eta(\Omega)=1$, then
\[
\pi_*\eta=\lambda.
\]
Consequently, $\eta$ is carried by $K\times\D^q$ and is singular with respect to $2n$-dimensional Lebesgue measure. In particular, no such competing measure is absolutely continuous with respect to volume; this also excludes Monge--Amp\`ere measures of $C^2$ psh functions.
\end{proposition}

\begin{proof}
For any negative subharmonic function $\chi$ on $E$, the function $(z,w)\mapsto\chi(z)$ is negative and psh on $\Omega$. By~\cite[Remark~2(1)]{Ng18}, it may be used in the order inequality. Hence
\[
\int_E(-\chi)\,d(\pi_*\eta)
\ge\int_E(-\chi)\,d\lambda.
\]
Both measures have mass one. The maximality of $\lambda$ yields $\pi_*\eta=\lambda$. Therefore $\eta((E\setminus K)\times\D^q)=0$. The remaining set has zero Lebesgue measure. The Monge--Amp\`ere measure of a $C^2$ psh function has a continuous density with respect to Lebesgue measure, and so cannot equal such a measure of mass one.
\end{proof}

\noindent\textbf{Acknowledgements}
This work was completed while the first author was visiting the Department of Mathematics at the National University of Singapore as a Visiting Senior Research Fellow under the Singapore Academies Southeast Asia Fellowship (SASEAF) Programme, within the project “Math, Sobolev Space” (Grant No. E-146-00-0039-01). The first author would like to express his sincere gratitude to the Department of Mathematics at NUS for its warm hospitality and excellent working environment. He also gratefully acknowledges the Singapore National Academy of Science (SNAS) and the National Research Foundation (NRF), Singapore, for their financial support through the SASEAF Programme.

\end{document}